\documentclass{amsart} 

\usepackage{amssymb,amsmath}
\usepackage{amsthm}

\newtheorem{definition}{Definition}[section]
\newtheorem{thm}{Theorem}[section]
\newtheorem{lemma}{Lemma}[section]
\newtheorem{proposition}{Proposition}[section]

\newtheorem{remark}{Remark}[section]

\def\bbc{\mathbb{C}}
\def\bbi{I}

\def\bbk{\mathbb{K}}
\def\bbl{\mathbb{L}}

\newcommand{\RR}{\mathbb{R}}
\newcommand{\KK}{\mathbb{K}}
\newcommand{\LL}{\mathbb{L}}

\newcommand{\continuum}{2^{\omega}}

\def\then{\rightarrow}

\def\ca{\mathcal{A}}

\newcommand{\B}[1]{{\rm B}[#1]}
\newcommand{\Borel}[1]{{\rm B}(#1)}
\newcommand{\Bplus}[1]{{\rm B}_{+}(#1)}

\def\<{\langle}
\def\>{\rangle}

\def\Cichon{Cicho{\'n}}
\def\Zeberski{{\.Z}eberski}
\def\Ralowski{Ra{\l}owski}

\begin{document}
\title{Remarks on nonmeasurable unions of big point families}
\author{Robert Ra{\l}owski}
\email{ralowski@im.pwr.wroc.pl}
\address{Institute of Mathematics and Computer Sciences, Wroc{\l}aw University of Technology, Wybrze\.ze Wyspia\'n\-skie\-go 27, 50-370 Wroc{\l}aw, Poland.}


\begin{abstract}
  We show that under some conditions on a family $\mathcal{A}\subset\bbi$ there exists a subfamily $\mathcal{A}_0\subset\mathcal{A}$ 
  such that  $\bigcup \mathcal{A}_0$ is nonmeasurable with respect to a fixed ideal $\bbi$ with Borel base of a fixed uncountable Polish space. 
  Our result applies to the classical ideal of null subsets of the real line and to the ideal of first category
  subsets of the real line.
\end{abstract}

\maketitle

\section{Introduction}
 We use standard set theoretical notations. Among others, we denote by $P(X)$ the family of all subsets of the set $X$.
We denote by $|X|$ the cardinality of the set $X$. We denote by $[X]^{<\omega}$ the set of all finite subsets of $X$, by $[X]^\omega$ the set of all countable subsets of $X$ and finally by $[X]^{\le\omega}$ the set of all at most countable subsets of $X$.
We denote by $\RR$ the real line. If $X$ is a topological space then we denote by $\Borel{X}$
the family of all Borel subsets of $X$.
Suppose that $I$ is a $\sigma$-ideal of subsets of $X$.
We denote by $\B{\bbi}$  the least $\sigma$-field containing
$\Borel{X} \cup \bbi$. Notice that $\B{I} = \{B \bigtriangleup A: B \in \Borel{X} \land A \in I\}$. We say that the set $A\subset X$ is measurable with respect to the ideal $I$ iff $A\in B[I]$.

Let $I$ be an ideal of subsets of a topological space $X$. 
We say that $I$ has a \textbf{Borel base}
if for any $A\in I$ there exists a Borel subset $B$ of $X$ such that $B \in I$ and $A \subseteq B$.
The ideal $\KK$ of the first category subsets of the real line $\RR$ has a Borel base, since every set
from $\KK$ can be covered by an $F_{\sigma}$ subset from $\KK$.
The ideal $\LL$ of Lebesgue measure subsets of the real line has a Borel base too, since every 
measure zero subset of the real line can be covered by a $G_{\delta}$ subset of measure zero.
The family $\B{\LL}$ is the family of Lebesgue measurable subsets of the real line
and $\B{\KK}$ coincides with the family of all subsets of the real line which has the Baire property.

\begin{definition}
 A pair $(X,\bbi)$ is a \textbf{Polish ideal space} 
 if $X$ is an uncountable Polish space, $\bbi\subseteq \mathrm{P}(X)$ is a proper $\sigma$-ideal 
 with a Borel base and $[X]^{\le \omega} \subseteq \bbi$. 
 A triple $(G,\bbi,+)$ is \textbf{Polish ideal group} if $(G,I)$ is a Polish ideal space, $(G,+)$ is abelian topological group and the ideal is translation invariant.
\end{definition}

Let $(X,\bbi)$ be a Polish ideal space. 
We denote by $\Bplus{\bbi}$ the family $\Borel{X}\setminus \bbi$. Notice that if $A \in \Bplus{\bbi}$ then $|A| = \continuum$. 

\begin{definition} 
  Let $\mathcal{X} = (X,\bbi)$ be a Polish ideal space. 
  We say that 
  a set $A \subseteq X$ is completely $\mathcal{X}$-nonmeasurable  if 
  $$
    (\forall B\in\Bplus{\bbi})( A\cap B\ne\emptyset\land A^c\cap B\ne\emptyset)~.
  $$
\end{definition}

Motivation of the above definition is as follows: let $A\subset \RR$ be any Lebesgue measurable set then we can find a two Borel sets let say $B_1,B_2\in\Borel{\RR}$ such that $B_1\subset A\subset B_2$ with $B_2\setminus B_1\in\bbl$. Then the inner Lebesgue measure is equal to outer Lebesgue measure. But in a case when $C\subset \RR$ is completely $(\RR,\bbl)$-nonmeasurable set then the set $C$ has the inner Lebesgue measure equal to $0$ and the set $C$ has full outer Lebesgue measure so the set $C$ is not Lebesgue measurable.

Let $X$ be an uncountable Polish space and let $\mathcal{X} = (X,[X]^{\le\omega})$.
Then $\mathcal{X}$ is a Polish ideal space.
A set $A \subseteq X$ is completely $\mathcal{X}$-nonmeasurable set if and only if 
$A$ is a classical Bernstein set.
Therefore the above definition is a generalisation of a classical property.

An ideal $I$ is c.c.c. if every family of pairwise disjoint non-empty $I$-positive Borel sets  is countable.
Now let $(X,I)$ be a Polish ideal space with $I$ c.c.c. and $A\subseteq X.$
Let $\mathcal{A}$ be a maximal family of pairwise disjoint $I$-positive Borel sets
contained in $A^c.$ Set $B=(\bigcup\mathcal{A})^c.$ Then $B$ is Borel, $A\subseteq B$ and for every Borel set $C\supseteq A$, $B\setminus C\in I.$ Any such set $B$ is called a {\it Borel envelope} of $A$ and will be denoted by $[A]_I.$ Note that a Borel envelope of $A$ is unique modulo $I$ and it is minimal (modulo $I$) Borel set containing $A.$

\begin{definition} Let $\mathcal{A}$ be a family of subsets of a Polish ideal space $(X,I)$. Then
\begin{enumerate}
\item $add( I )   = \min\{ |\mathcal{C}|:\;\mathcal{C}\subset  I \land \bigcup \mathcal{C}\notin  I  \}$.
\item $cov(\mathcal{A})   = \min\{ |\mathcal{B}|:\;\mathcal{B}\subset \mathcal{A}\land \bigcup \mathcal{B}=X \}$,
\item $cov_h(\mathcal{A}) = \min\{ |\mathcal{B}|:\;\mathcal{B}\subset \mathcal{A}\land (\exists B\in \Bplus{X})( B\subset \bigcup \mathcal{B}) \}$.
\end{enumerate}
\end{definition}

Notice that if $I$ is an ideal and $\mathcal{A} \subseteq I$ then $add(I) \leq cov_h(\mathcal{A})  \leq cov(\mathcal{A})$.
It is easy to see that  $cov_h(\KK)=cov(\mathcal{\KK})$ and $cov_h(\LL)=cov(\mathcal{\LL})$.

Suppose that $(G,+)$ is and abelian group $A,B\subseteq G$ and $g\in G$. Then
we put $A+g = \{a+g:a \in A\}$ and $A+B=\{ a+b \in G:\;\; a\in A \land b \in B\}$.
We call the set $A+B$ the algebraic sum of the sets $A$ and $B$.


\section{Nonmeasurable unions of null sets}

Suppose that $\mathcal{A}$ is a family of subsets of a set $X$ and let $x \in X$. We put
$\mathcal{A}_x = \{A \in \mathcal{A}: x\in A\}$.
We say that a family $\mathcal{A} \subseteq P(X) $ is \textbf{point-finite} if $|\mathcal{A}_x| < \omega$
for each $x \in X$.
The classical  Four Poles Theorem (see \cite{fourpoles}) says that if $(X,I)$ is a Polish ideal space
and $\mathcal{A} \subseteq I$ is a  point-finite family such that $\bigcup \mathcal{A} = X$ then
there exists a subfamily $\mathcal{B} \subseteq \mathcal{A}$ that $\bigcup \mathcal{B}$
is $(X,I)$-nonmeasurable.
The problem of the existence of subfamilies with a completely nonmeasurable unions was discussed in \cite{zeberski}, where a general theorem was proved for ccc $\sigma$-ideals with Borel base under the assumption of the non existence of quasi-measurable cardinal less or equal than continuum. 
In the ZFC theory the above problem was solved under some assumption on regularity of the family $\mathcal{A}$ (see \cite{twopoles}). 
The complete nonmeasurability of the union of ,,small-point'' family of subset was investigated in \cite{fivepoles}.
In this section we consider the unions of ,,big-point'' families of the subsets of Polish ideal spaces.

\begin{thm}\label{tw1} 
Let $(X,I)$ be a Polish ideal space. Suppose that a family $\ca   \subseteq I$ has the following properties:
\begin{enumerate}
\item $(\forall x\in X)( |\ca_x|=\continuum)$,
\item $(\forall x,y\in X$) ( $x\ne y\then |\ca_x \cap \ca_y|\le\omega)$,
\item $cov_h(\mathcal{A})=\continuum .$
\end{enumerate}
Then there exists a subfamily $\ca_0\subseteq \ca$ such that $\bigcup \ca_0$ is a completely $(X,I)$-nonmeasurable set.
\end{thm}

\begin{proof} Let $\Bplus{I} = \{ B_\alpha:\alpha<\continuum \}$.
We will build a sequence 
$ ((A_\xi,d_\xi))_{\xi<\continuum}$ such that for all $\xi < \continuum$ we have $A_\xi \in \ca$, $d_\xi \in B_\xi$ and
\begin{enumerate}
\item  $(\forall \xi< \continuum)(A_\xi\cap B_\xi\ne\emptyset)$,
\item  $\{ d_\xi:\xi< \continuum \} \cap \bigcup_{\xi<\continuum} A_\xi=\emptyset$.
\end{enumerate}
Suppose that $\alpha < \continuum$, let $D_\alpha=\{ d_\xi:\;\; \xi<\alpha\}$ and that we have built a sequence $ ((A_\xi,d_\xi))_{\xi<\alpha}$
such that for all $\xi < \alpha$ we have $A_\xi \in \ca$, $d_\xi \in B_\xi$,
$(\forall \xi< \alpha)(A_\xi\cap B_\xi\ne\emptyset)$ and 
$\{ d_\xi:\xi< \alpha \} \cap \bigcup_{\xi<\alpha} A_\xi=\emptyset$.
Let $d_\alpha\in B_\alpha\setminus \bigcup_{\xi<\alpha} A_\xi$ what is possible by assumption (3) and $x_0\in B_\alpha\setminus(D_\alpha\cup\{ d_\alpha\})$. From assumption (2) we get
$$
  (\forall \xi\le \alpha) (|\ca_{x_0}\cap \ca_{d_\xi}|\le \omega) ,
$$
so, from assumption (1), we get
$$
  \ca_{x_0}\setminus \bigcup_{\xi\le\alpha}(\ca_{x_0}\cap\ca_{d_\xi})\ne\emptyset.
$$
Let us fix any  $A_\alpha \in \ca_{x_0}\setminus \bigcup_{\xi\le \alpha}(\ca_{x_0}\cap\ca_{d_\xi})$.
Then  $\{d_\xi:\xi\leq \alpha\} \cap A_\alpha = \emptyset$.
This show that there exists a sequence satisfying properties (1) and (2).

Finally, let us put $\ca_0=\{ A_\xi\in \ca:\;\;\xi<\continuum \}$. 
Let $B$ be any set from $\Bplus{X}$. Then there exists $\xi<\continuum$ such that
$B = B_\xi$. From (1) we deduce that $B_\xi \cap \bigcup \mathcal{A}_0 \neq \emptyset$
and from (2) we deduce that $B_\xi \setminus \bigcup \mathcal{A}_0 \neq \emptyset$.
Therefore $\bigcup \mathcal{A}_0$ is a completely $(X,I)$-nonmeasurable set.
\end{proof}


\begin{thm}\label{niemierzalne} 
Let $(X,I)$ be a Polish ideal space. Suppose that a family $\mathcal{A}\subset  I $ satisfies the following conditions:
\begin{enumerate}
\item $\bigcup\mathcal{A}=X$,
\item $(\forall x,y\in X)( x \neq y  \then |A_{x}\cap\mathcal{A}_{y}|  \le \omega )$,
\item $cov_h(\mathcal{A})=\continuum$. 
\end{enumerate}
Then there exists a subfamily $\mathcal{A}_0\subset\mathcal{A}$ such that $\bigcup\mathcal{A}_0$ is $(X,I)$-nonmeasurable.
\end{thm}

\begin{proof} 
Let $B_+(I)=\{ B_\xi:\;\xi<\continuum\}$.
For $Z\subset X$ we put $\mathcal{A}_Z= \{ A\in\mathcal{A}:A\cap Z\ne\emptyset\}$. 
Let us consider the two alternative 
possibilities:
\begin{description}
\item[P1] $(\forall D\subset X)(|D|<\continuum \then \bigcup\mathcal{A}_D\in    I) $,
\item[P2] $(\exists D\subset X)(|D|<\continuum \land \bigcup\mathcal{A}_D\notin I) $.
\end{description}

{\bf Case  P1}. 
We will construct by the transfinite induction a sequence $((A_\alpha,d_\alpha))_{\alpha<\continuum}$ with the following properties:
\begin{enumerate}
\item $(\forall \alpha < \continuum)(A_\alpha,d_\alpha)\in (\mathcal{A}\times B_\alpha)$,
\item $(\forall \alpha < \continuum)(A_\alpha  \cap B_\alpha \ne \emptyset)$,
\item $(\forall \alpha < \continuum)(\{d_\xi:\;\xi<\alpha\}\cap\bigcup_{\xi<\alpha} A_\xi=\emptyset)$.
\end{enumerate}
Suppose that $\alpha < \continuum$ and that we have constructed a partial sequence $((A_\xi,d_\xi))_{\xi<\alpha}$
with satisfies the above conditions (restricted to $\alpha$).
Let $D_\alpha = \{ d_\xi:\;\xi<\alpha\}$. 
Then $|D_\alpha|\le |\alpha| <\continuum$. Therefore by the condition {\bf P1} we have
$\bigcup\mathcal{A}_{D_\alpha}\in I$.
Let us fix $x_0\in B_\alpha\setminus\bigcup\mathcal{A}_{D_\alpha}\ne\emptyset$. 
Then $\mathcal{A}_{x_0}\cap\mathcal{A}_{D_\alpha}=\emptyset$ and $\mathcal{A}_{x_0}\ne\emptyset$. 
Let us choose any $A_\alpha\in\mathcal{A}_{x_0}$. Notice that $D_\alpha\cap A_\alpha=\emptyset$. 
Finally, let us choose $d_\alpha\in B_\alpha\setminus\bigcup_{\xi<\alpha+1} A_\xi$.
Then the sequence $((A_\xi,d_\xi))_{\xi<\alpha+1}$ also satisfies the the above three conditions.
This show that a required sequence exists. 
It is easy to check that the union of the family $\{A_\alpha:\alpha < \continuum\}$ is $(X,I)$-nonmeasurable.

{\bf Case P2}.
Let $D\subseteq X$ be a subset of $X$ such that $|D|<\continuum$ and $\bigcup\mathcal{A}_D\notin I $. 
If    $\bigcup\mathcal{A}_D$ is $(X,I)$-nonmeasurable then proof is finished. 
Suppose hence that the set  $\bigcup\mathcal{A}_D$ is measurable.
The there are two possibilities: $D \not \in I$ or $D\in I$.

Suppose first that $D\notin I$. Then $D$ is a nonmeasurable set, since otherwise $D$ would contain some perfect set, which is impossible ($|D|<\continuum$). 
For each $d\in D$ we choose some $A_d \in \mathcal{A}$ such that $d \in A_d$ and we put 
$\mathcal{A}_0 = \{A_d:d\in D\}$.
Then $D \subseteq \bigcup\mathcal{A}_0$, so $\bigcup\mathcal{A}_0 \notin I$.
Moreover $|\mathcal{A}_0|\le |D|<\continuum=cov_h(\mathcal{A})$, so for every
$B \in B_+(I)$ we have 
$B\setminus\bigcup\mathcal{A}_0\ne\emptyset$.
This implies that the set $\bigcup\mathcal{A}_0$ is $(X,I)$-nonmeasurable.

Suppose now that $D\in I $. Without loss of generality we may assume that $\mathcal{A}=\mathcal{A}_D$. 
Let $Z=\bigcup\mathcal{A}_D$ and let $\{ B_\xi:\xi<\continuum\} = (P(Z) \cap B(X)) \setminus I$.
We shall build a sequence $((A_\alpha,d_\alpha))_{\alpha<\continuum}$ such that:
\begin{enumerate}
\item $(\forall \xi<\continuum)(A_\xi \in \mathcal{A} \land d_\xi \in B_\xi \setminus D)$,
\item $(\forall \xi<\continuum)(A_\xi\cap B_\xi\ne\emptyset)$,
\item $(\forall \xi <\continuum)(\{ d_\beta:\beta<\xi\}\cap \bigcup_{\beta<\xi} A_\beta=\emptyset)$.
\end{enumerate}
Suppose that $\alpha<\continuum$ and that a sequence $((A_\xi,d_\xi))_{\xi<\alpha}$ satisfies
the above three conditions.
Let us observe that for any $\xi<\alpha$ we have $|\mathcal{A}_{d_\xi}|\le |\bigcup_{d\in D}\mathcal{A}_d\cap\mathcal{A}_{d_\xi}|\le |D|$,
so  
$|\mathcal{A}_{\{ d_\xi:\xi<\alpha\} }|\le |\alpha|\cdot |D|< \continuum$. 
Therefore  $B_\alpha\setminus (D\cup\bigcup\bigcup_{\xi<\alpha}  \mathcal{A}_{d_\xi})\ne \emptyset$.
Let us choose any 
$x_0\in B_\alpha\setminus (D\cup\bigcup\bigcup_{\xi<\alpha}  \mathcal{A}_{d_\xi})$ and 
$A_\alpha\in\mathcal{A}_{x_0}$. Then
$A_\alpha\cap \{ d_\xi:\; \xi<\alpha\}=\emptyset$ and  $\{d_\xi:\;\xi<\alpha\}\cap\bigcup_{\xi<\alpha+1} A_\xi=\emptyset$.
Finally, let us choose any $d_\alpha\in B_\alpha\setminus (D\cup\bigcup_{\xi<\alpha+1} A_\xi)$.
Then the sequence $((A_\xi,d_\xi))_{\xi<\alpha+1}$ satisfies the three inductive assumptions.
Therefore the sequence $((A_\alpha,d_\alpha))_{\alpha<\continuum}$ with the above three properties exists.
It is easy to check that  
the union of the family $\{A_\alpha:\alpha < \continuum\}$ is $(X,I)$-nonmeasurable. 
\end{proof}



\section{Nonmeasurable algebraic sums of null sets on Polish groups}

In this section we will consider nonmeasurability in abelian Polish groups. Namely, the union of the translations of some subset from some fixed $\sigma$ - ideal with Borel base which contains all singletons.

It is well known that so called Four Poles Theorem see \cite{fourpoles} is in some sense the best we can get in ZFC. In case of families that are not point-finite the complete $(X,I)$-nonmeasurability is independent of ZFC theory including countable point families also see \cite{fremlin}.  Then in general case we need additional set theoretic assumptions as in the paper \cite{fivepoles} for example:

\begin{thm}[$ZFC+cov( I )=\continuum$]\label{stare} Let $(X, I )$ be a Polish ideal space and let $\mathcal{A}\subset  I$ be such that:
\begin{enumerate}
\item $\bigcup\mathcal{A}=X$,
\item $\{ x\in X:\;\bigcup\mathcal{A}_x\notin  I \}\in  I $.
\end{enumerate}
Then there exists subfamily $\mathcal{A}_0\subset\mathcal{A}$ such that union $\bigcup\mathcal{A}_0$ is completely $(X,I)$-nonmeasurable.
\end{thm}

But in special cases if we adopt some regularity conditions on our families then the results on nonmeasurability can be proved
in the theory ZFC (see \cite{twopoles}). In this section we will introduce the notion of tiny perfect set with respect to a family of subsets of the Polish space. This gives some regularity of the family of subsets $\mathcal{A}$ and finally $cov_h(\mathcal{A})=\continuum$.

Here we are would like present the probably well known lemma on translation. The proof of this lemma was inspired by Ryll-Nardzewski.

\begin{lemma}\label{waznylemat} 
Let $(G,\bbl)$ be a Polish ideal group with a Haar measure $\lambda$ and let $Perf(G)$ be a family of all perfect sets in $G$.
Then
$$
  (\forall P\in Perf(G))(\forall B\in\Bplus{G})(\exists x_0\in G)(|(P+x_0)\cap B|=\continuum).
$$
Moreover, if $P\in Perf(G)$ and $B\in\Bplus{G}$ then
$$
\lambda(\{ x\in G:\; |(P+ x_0)\cap B|=\continuum \})>0.
$$
\end{lemma}

\begin{proof} 
Let $\mu$ be a regular measure defined on perfect set $P$ such that $\mu(P)=1$ and let 
$$
D = (G\times P)\cap (\bigcup\limits_{t\in\RR} \{ t\}\times (t + B))\subset G^2,
$$
where $B$ is a Borel set such that $\lambda(B)>0$. Let us observe that 
$$
  (t,s)\in D  \equiv  (s\in P) \land (\exists b\in B)(t+b=s ) \equiv (s\in P)\land (\exists b\in B)(t=s-b) ~,
$$
therefore
$$
  D=\bigcup_{s\in P} (\bigcup_{b\in B} s-b)\times \{ s\}=\bigcup_{s\in P} (s-B)\times \{ s\} ~.
$$
Let us consider the product measure $\lambda\times\mu$. Then we have:
\begin{align}
(\lambda\times\mu) (D)=&\int\limits_D 1\;d(\lambda\times\mu)(t,s)=\int\limits_P d(\mu)(s)\int\limits_{s+ B} d\lambda(t)=\int\limits_P d(\mu) \lambda(s+B)\\
=&\int\limits_P d(\mu)\lambda(B)=\lambda(B)\int\limits_P d(\mu)=\lambda(B)\mu(P)=\lambda(B)>0.
\end{align}
Then, by Egglestone theorem (see \cite{egglestone} and \cite{zeberski1})  
there are two perfect sets $P_1, P_2$ such that $\lambda(P_1)>0$ and $P_1\times P_2\subset D$. Note that
$$
  t\in P_1\then P_2\subset P\cap (t+B),
$$
so $\lambda(\{ t\in G:\;\;|P\cap (t+B)|=\continuum\})>0$.
\end{proof}

\begin{remark} 
  Let us note that there are other proofs of the above lemma. 
  First one was given by by Cicho{\'n} and uses the Shoenfield absoluteness argument. 
  The second one was due to Ryll-Nardzewski and is based on convolution measures and was an inspiration for the proof presented above. 
  Another one was due to Morayne, where density point of measure was used.
\end{remark}

\begin{definition} 
  Let $(G,I)$ be a Polish ideal group, and let $\mathcal{A}\subset\mathcal{P}(G)$. 
  A perfect set $P\subset G$ is a {\bf tiny perfect set with respect to $\mathcal{A}$} if
  $$
  (\forall t\in G)(\forall A\in\mathcal{A})( |(\{ t\}+P)\cap A|\le\omega ) .
  $$
\end{definition}

\begin{lemma}\label{covering} 
  Let $(G,I)$ be a Polish ideal group and let $\mathcal{A}\subset I$. Suppose that there exists a tiny set with respect to the family $\mathcal{A}$.
  Then $cov_h(\mathcal{A})=\continuum$.
\end{lemma}

\begin{proof} 
  Let us consider any $I$ positive Borel set $B\in B_+(I)$ and any tiny perfect 
  set $P$ with respect to our family $\mathcal{A}\subset I$. Then there exists a translation $x_0+P$ of $P$ for some $x_0\in G$ such that
  $|(P+x_0)\cap B|=\continuum$ by Lemma \ref{waznylemat}. Let us choose any subfamily $\mathcal{A}_0\subset\mathcal{A}$ with $\kappa=|\mathcal{A}_0|<\continuum$. But for any $A\in\mathcal{A}_0$ $|A\cap (x_0+P)|\le\omega$ then
  $|\bigcup\mathcal{A}_0\cap (x_0+P)|\le \kappa\cdot\omega<\continuum$ and we have
  $(B\cap (x_0+P)\setminus\bigcup\mathcal{A}_0\ne\emptyset$.
\end{proof}

\begin{proposition}\label{prop1} 
  Let $(G,I)$ be a Polish abelian group and let $\mathcal{A}\subset  I $ be such that
  \begin{enumerate}
     \item $(\forall x\in G)(|\mathcal{A}_x|=\continuum)$,
     \item $(\forall x,y\in G)(x\ne y \then |\mathcal{A}_x\cap \mathcal{A}_y|\le\omega)$,
     \item there exists tiny perfect set $P$ with respect to the family $\mathcal{A}$.
  \end{enumerate}
  Then there exists subfamily $\mathcal{A}_0\subset \mathcal{A}$ 
  such that $\bigcup\mathcal{A}_0$ is completely $(G,I)$-nonmeasurable set in $G$.
\end{proposition}

\begin{proof} 
  The result follows from Theorem \ref{tw1} and Lemma \ref{covering}.
\end{proof}

\begin{proposition}\label{prop2} 
  Let $(G,I)$ be a Polish abelian group and let $\mathcal{A}\subset  I $ be such that
  \begin{enumerate}
  \item $\bigcup\ca=G$,
  \item $(\forall x,y\in G)(x\ne y\then |\mathcal{A}_x\cap \mathcal{A}_y|\le\omega)$,
  \item there exists tiny perfect set with respect to the family $\mathcal{A}$.
\end{enumerate}
Then there exists subfamily $\mathcal{A}_0\subset \mathcal{A}$ such that $\bigcup\mathcal{A}_0$ is $(G,I)$-nonmeasurable.
\end{proposition}
\begin{proof} 
  The result follows from Theorem \ref{niemierzalne} and Lemma \ref{covering}.
\end{proof}

We present some applications of the above Propositions.

\begin{thm}\label{proste} 
  Let $2\le n\in\omega$ be positive integer. Then there exists a family $\mathcal{L}$ of lines in $\RR^n$ 
  such that $\bigcup \mathcal{L}$ is a completely $(\RR^n,\bbl)$-nonmeasurable set.
\end{thm}
\begin{proof} 
Let $\mathcal{A}$ be family of all lines in $\RR^n$. Then for all $x,y\in\RR^n$ if $x\ne y$ then 
$$
|\{ l\in\mathcal{A}:\; x\in l\}\cap \{ l\in\mathcal{A}:\; y\in l\}|=1\leq\omega\text{ and }
|\{ l\in\mathcal{A}:\; x\in l\}| =\continuum.
$$
Let us observe that the unit Euclidean sphere $S$ is a tiny perfect set with respect to the family $\mathcal{A}$. 
So,  we get the theorem from Proposition \ref{prop1}.
\end{proof}

\begin{remark}[Given by referee] Even the a stronger statement (with parallel lines) is simple: if $A\subset \RR^{n-1}$ is completely nonmeasurable, then so is $A\times \RR$.
\end{remark}

\begin{thm}\label{okregi} 
  Let $\mathcal{A}$ be a family of circles with radius 1 on the plane $\RR^2$ such that $\bigcup\mathcal{A}=\RR^2$.
  Then there exists $\mathcal{A}_0\subset\mathcal{A}$ such that $\bigcup\mathcal{A}_0$ is a Lebesgue nonmeasurable (do not have the Baire property). Moreover, if additionally every point is covered $2^\omega$ many times then there exists $\mathcal{A}_0\subset\mathcal{A}$ such that $\bigcup\mathcal{A}_0$ completely $(\RR^2,\bbi)$-nonmeasurable set, where $\bbi=\bbl$ or $\bbi=\bbk$.
\end{thm}
\begin{proof} 
  The proof is analogous to the previous one, but here we use Proposition \ref{prop2} for the first statement.
\end{proof}

In the above example we have the situation where our family is a set of one dimensional circles. 
But under some additional assumptions we can get the assertion for arbitrary finite dimension of the spheres. 
Unfortunately here we cannot use our Propositions so proof will be a bit longer.

\begin{thm}\label{miara} 
  Let $n\in\omega$ be a fixed positive integer and let us us consider a
  family of $n-1$ dimensional Euclidean spheres 
  $\mathcal{A}\subset \{S(x,r):\;x\in\RR^n\land r>0\}$ with the following property:
$$
\forall x\in\RR^n\;\; \{ y\in\RR^n:\;\exists r>0\;\; x\in S(y,r)\in\mathcal{A}\}\in\B{\bbl}\setminus\bbl, 
$$
then there exists $\mathcal{A}_0\subset \mathcal{A}$ such that $\bigcup\mathcal{A}_0$ is completely $(\RR^n,\bbl)$-nonmeasurable.
\end{thm}

\begin{proof} 
Let us enumerate the set of all positive Borel sets $B_+(\LL)=\{ B_\alpha:\; \alpha<\continuum\}$ and for any $\alpha$ let us assume that we have defined a transfinite sequence:
$
\< (A_\xi,d_\xi)\in\mathcal{A}\times B_\xi:\;\;\xi<\alpha\>
$
with the following conditions:
\begin{enumerate}
  \item $(\forall\xi<\alpha)(A_\xi\cap B_\xi\ne \emptyset)$,
  \item $\{ d_\xi:\;\xi<\alpha\}\cap\bigcup \{A_\xi:\;\xi<\alpha\}=\emptyset$.
\end{enumerate}
Now let us choose any $x\in B_\alpha\setminus\{ d_\xi:\;\xi<\alpha\}$. 
For every $\xi<\alpha$ let $H_\xi$ be the perpendicular bisector hyperplane of the segment connecting $x$ and $d_\xi$. Let us consider the set
$$
\mathcal{H}=\{ H_\xi\subset\RR^n:\;\; \xi<\alpha \}.
$$
By Lemma \ref{covering} a Lebesgue positive set cannot be covered by $\bigcup\mathcal{H}$.
Then there exists a positive real $r>0$ and $y\in\RR^n$ such that $x\in S(y,r)\in\mathcal{A}$ and $S(y,r)\cap\{ d_\xi:\;\xi<\alpha\}=\emptyset$.

Since our family $\mathcal{A}$ is a tiny with respect to any line in $\RR^n$, by Lemma \ref{covering} there exists $d\in B_\alpha\setminus \bigcup_{\xi<\alpha+1} A_\xi$ and let $d_\alpha=d$. 
So we have built a sequence 
$
\< (A_\xi,d_\xi)\in\mathcal{A}\times B_\xi:\;\;\xi<\alpha+1\>
$
of length $\alpha+1$ which satisfies the following conditions:
\begin{enumerate}
\item $(\forall\xi<\alpha+1)(A_\xi\cap B_\xi\ne \emptyset)$,
\item $\{ d_\xi:\;\xi<\alpha+1\}\cap\bigcup{A_\xi:\;\xi<\alpha}=\emptyset$.
\end{enumerate}
Therefore, by the transfinite induction, we can have the analogous sequence with the length $\continuum$ 
and then putting $\mathcal{A}_0=\{ A_\xi:\;\xi<\continuum\}$ we get the assertion.
\end{proof}

From now we will consider nonmeasurability of Cantor-like sets where crucial role is played by packing dimension which is closely related to coverings of the real line see \cite{elekes_steprans},\cite{elekes_toth},\cite{darji}. 

It was proved in \cite{fivepoles} that there exists subset $A$ of the classical Cantor set $\bbc$ 
such that $A+\bbc$ is completely $([0,2],\bbl)$-nonmeasurable. This proof was done by the ultrafilter method. 
Here we give a new proof of this theorem under some additional set theoretical assumption.

\begin{thm}[$ZFC+cov(\mathcal{L})=\continuum$]\label{thm1} 
  Let $\mathcal{R}=\{ 0,1,2\}^\omega$ be the additive group with coordinatewise addition $\!\! \mod 3$ and let $\mathcal{C}_3 = \{0,2\}^\omega$ be the Cantor-like set. 
  Then there exists a set $A\subset\mathcal{C}_3$ such that $A+\mathcal{C}_3$ is completely $(\mathcal{R},\mathcal{L})$-nonmeasurable set.
\end{thm}

\begin{proof} 
  Notice that $\mathcal{C}_3+\mathcal{C}_3=\mathcal{R}$. Moreover $\mathcal{C}_3$ 
  and their translations are null sets. Therefore we have to check that the second  condition in Theorem \ref{stare} holds. 
  Let us observe that $H=\{ x\in\mathcal{R}:\; |\{ n\in\omega:\; x(n)=0\}|<\omega\}\in\mathcal{L}$. 
  Let $\mathcal{A}=\{ t+\mathcal{C}_3:\; t\in\mathcal{C}_3\}$ and $D(x)=\{ n\in\omega:\; x(n)=0\}$. 
  Let $x\in \mathcal{R}\setminus H$ and let us suppose that $t\in\mathcal{C}_3$ is such that $x\in t+\mathcal{C}_3$. 
  Then for every $n\in D(x)$ we have $t(n)=0$ ($2+0=2$ and $2+2=1$) and for every $c\in\mathcal{C}_3$ we have $t(n)+c(n)\in\{ 0,2\}$. 
  But $|D(x)|=\omega$ for every $x\in\mathcal{R}\setminus H$, so $\{ y\in\mathcal{R}:\; \forall n\in D(x)\; y(n)\in\{ 0,2\} \}\in\mathcal{L}$, 
  hence
  $$
    \bigcup \mathcal{A}_x\subset \{ y\in\mathcal{R}:\; \forall n\in D(x)\; y(n)\in\{ 0,2\} \}\in\mathcal{L}
   $$
  for every $x\in\mathcal{R}\setminus H$. 
  Let us observe that 
  $\mathcal{R}\setminus H\subset \bigcup_{x\in\mathcal{R}\setminus H}\mathcal{A}_x$ 
  and that the set $\bigcup_{x\in\mathcal{R}\setminus H}\mathcal{A}_x$ has a full measure. 
  So the second condition of the Theorem \ref{stare} is satisfied, which finishes the proof.
\end{proof}

Using the same method we can prove an analogous theorem for the classical Cantor subset of the real line
in the theory $ZFC+cov(\bbl)=\continuum$.

\begin{thm}[Darji, Keleti \cite{darji}] 
  If a set $C\subset\RR$ has packing dimension 
  less than $1$ then for any positive measure Borel set $B$ 
  and any $T\subset\RR$ such that $|T|<\continuum$ we have $B\setminus (T+C)\ne\emptyset$.
\end{thm}

The above theorem is a generalisation of Gruenhage's theorem about the classical  Cantor set. 
It is important to note that assumption on packing dimension is necessary (see \cite{elekes_steprans}, \cite{elekes_toth}).

Let $\mathcal{C}_4\subset[0,1]$ be a Cantor-like set constructed  as follows: 
\begin{enumerate}
 \item remove from the interval $[0,1]$ the open segment $(\frac{1}{4},\frac{3}{4})$,
 \item do the same with the remaining segments infinitely many times.
\end{enumerate}

It is easy to check that $\dim_{pack} (\mathcal{C}_4)=\frac{1}{2}$, 
thus by Darji-Keleti Theorem (see \cite{darji}) we need continuum many translates of 
the set $\mathcal{C}_4$ to cover any set of positive Lebesgue measure.

\begin{proposition}[ZFC+$cov(\bbl)=2^\omega$]
  There exists a set $A\subset \frac{1}{2}\mathcal{C}_4$ such that $A+\mathcal{C}_4$ 
  is a $([0,\frac{3}{2}],\bbl)$-completely nonmeasurable subset of the interval $[0,\frac{3}{2}]$.
\end{proposition}

\begin{proof}({\it Sketch}) 
First of all one can prove that $\mathcal{C}_4+\frac{1}{2}\mathcal{C}_4=[0,\frac{3}{2}]$ 
and that for any $x\in [0,\frac{3}{2}]$ the set $\{ (a,b)\in \mathcal{C}_4\times \frac{1}{2}\mathcal{C}_4: x=a+b\}$ is finite
(see \cite{ciesielski} Lemma 7 for details). 
Let $\mathcal{A}=\{ t+\mathcal{C}_4:\; t\in\frac{1}{2}\mathcal{C}_4 \}\subset\bbl$ then $\mathcal{A}$ is point-finite family.
Therefore, by Theorem \ref{stare}, 
there exists $\mathcal{A}_0\subset\mathcal{A}$ 
such that $\bigcup \mathcal{A}_0$ is completely $([0,\frac{3}{2}],\bbl)$-nonmeasurable in the interval $[0,\frac{3}{2}]$.
\end{proof}

\noindent{\bf Acknowledgement} 
Author is very indebted to Professor Ryll-Nardzewski for an idea of the proof of Lemma \ref{waznylemat}, 
Professors Jacek Cicho{\'n} and Micha{\l} Morayne for critical remarks. Author would like to thank the referee for finding serious mistakes in an earlier version of this work and for giving valuable suggestions.


\vskip 1cm

\end{document}